\documentclass[12pt,a4paper]{article}

\usepackage[utf8]{inputenc}
\usepackage{amsmath,amssymb,amsthm}
\usepackage{graphicx}
\usepackage{tikz}
\usepackage{subcaption}
\usepackage[numbers]{natbib}
\usepackage{float}
\usepackage{hyperref}
\usepackage{enumitem}
\newtheorem{theorem}{Theorem}

\newtheorem{corollary}{Corollary}

\theoremstyle{definition}
\newtheorem{definition}{Definition}
\newtheorem{example}{Example}

\theoremstyle{remark}
\newtheorem{remark}{Remark}

\title{On the Transversal Coalition in $r$-Uniform Hypergraphs}
\author{
Vrushali Shinde$^{1}$,
Lata Kadam$^{2}$\thanks{Corresponding author}
}
\begin{document}
\maketitle
\vspace{-10mm}
\begin{center}
\small
Department of Engineering Science and Humanities, 
Marathwada Mitra Mandal’s College of Engineering,
Pune, India.\\
\texttt{vrushali.kadam85@gmail.com}
\vspace{2mm}

$^{2}$Department of Mathematics,
M.E.S’s Abasaheb Garware College (Autonomous),
Pune, India.\\
\texttt{lata7429@gmail.com}
\vspace{2mm}

\end{center}
\begin{abstract}
 A transversal coalition in a hypergraph $H$ is a partition of the vertex set $U$ into two subsets $U_1$ and $U_2$ such that neither $U_1$ nor $U_2$ intersects every hyperedge of $H$. However, their union, $U_1 \cup U_2$, intersects every hyperedge in $H$. In this work, we investigate $trc$-partitions in $r$-uniform hypergraphs. Specifically, we determine the transversal coalition number for complete $r$-uniform hypergraphs, complete bipartite $r$-uniform hypergraphs, $r$-uniform stars, and complete $r$-partite $r$-uniform hypergraphs. We also investigated the transversal coalition number of $r$-uniform linear paths and cycles.
\end{abstract}
\section{Introduction}
\label{sec:in}
Domination-related vertex partitions in graphs have been widely investigated in the literature. The notions of dominating sets and domination number were first formally introduced by Øystein Ore in 1962 \cite{Ore}. A dominating set in a graph is a set of vertices such that every vertex not in the set is adjacent to at least one vertex in the set. Since then, domination theory has been extended beyond graphs, and the study of domination in hypergraphs has received considerable attention. Domination in both graphs and hypergraphs has been extensively studied in the literature, as reviewed in \cite{Ach, Div, Hay1}. In a graph $G$, a coalition consists of two disjoint vertex subsets $X$ and $Y$ such that neither $X$ nor $Y$ is a dominating set of $G$, while their union $X \cup Y$ is a dominating set. In this situation, the sets $X$ and $Y$ are said to form a coalition and are called coalition partners. A coalition partition of a graph $G$ is a partition $\mathcal{P}=\{U_1, U_2, \ldots, U_k\}$ of $U(G)$ such that each part $U_i$ consists of a single vertex that dominates $G$, or there exists a distinct part $U_j$ with which $U_i$ forms a coalition. The coalition number of a graph $G$, denoted by $C(G)$, is the maximum integer $k$ for which $G$ admits a coalition partition with $k$ parts.\\
\indent By coalition, we mean a temporary collaboration between two or more groups aimed at attaining a mutual goal. Coalitions in graphs were first introduced and studied by T. W. Haynes et al. \cite{Hay}. 
Since then, this concept has evolved in several directions. 
Alikhani, Bakhshesh, and Golmohammadi explored total coalitions \cite{Ali}. Subsequently, Golmohammadi et al. \cite{Golm} investigated strong coalitions, while Henning and Mojdeh \cite{Hen} extended the investigation to double coalitions. Furthermore, Mojdeh \cite{Moj, Moj1} contributed by studying perfect and edge coalitions, and independent coalitions were addressed by Samadzadeh and Mojdeh \cite{Sam}.\\
\indent  The transversal coalition in hypergraphs and the transversal coalition number were studied in full generality by Henning et al. \cite{Hen1}, who established sharp upper bounds on the transversal coalition number and showed, via constructions based on Latin squares, that these bounds are tight. They also introduced the notion of transversal coalition graphs. In this paper, we restrict our attention to specific classes of hypergraphs, where we obtain sharper bounds and structural insights.\\
\indent A transversal coalition is a group of vertices that intersects all edges in a hypergraph, modeling how control or resources are shared. In $r$-uniform hypergraphs, this concept applies to resource allocation, network security, social networks, and cooperative games.\\
\indent We focus on coalitions involving transversal sets in $r$-uniform hypergraphs. Although general results provide key theoretical insights, special cases such as complete $r$-uniform hypergraphs, complete bipartite $r$-uniform hypergraphs, $r$-uniform linear paths and cycles reveal unique structural properties that refine these general theories. These special cases also have practical applications in network design, resource allocation, and scheduling, highlighting their importance for both theory and real-world optimization problems.\\
\indent Consider optimization of sensor networks for environmental monitoring. Sensor networks are vital for critical applications, including the monitoring of wildfires, humidity, and animal movement. To ensure reliability and accuracy, certain regions may require simultaneous coverage by multiple sensors. This scenario can be effectively modeled as an $r$-uniform hypergraph, where sensors are vertices and regions requiring coverage are hyperedges of size $r$. The challenge lies in identifying a \textbf{minimal transversal coalition}, the smallest subset of sensors that collectively monitors all critical regions, thus minimizing maintenance and energy costs.\\
\indent Consider a forest monitoring system deployed with five sensors $V=\{S_1, S_2,\dots, S_5\}$ covering four critical regions, modeled as a $3$-uniform hypergraph with hyperedges $E=\{\{S_1, S_2, S_3\}, \{S_2, S_3, S_4\}, \{S_3, S_4, S_5\}, \{S_1, S_4, \\S_5\}\}$. Here, the set $\{S_3, S_4\}$ forms a minimal transversal, ensuring all regions are monitored by activating only two sensors. This practical example underscores the significance of transversal coalitions and motivates the investigation of their properties within specific hypergraph families.\\
\indent In practical sensor networks, ensuring complete coverage with minimal resources is a key challenge. Studying \emph{transversal coalitions} helps us understand how groups of sensors can collectively ensure coverage, even when individual subsets are insufficient. While general results exist for arbitrary hypergraphs, real-world sensor networks often have structured layouts that permit more precise analysis. \\
\indent Sensor deployments can be modeled using $r$-uniform hypergraphs, providing a useful framework for analyzing cooperative coverage. Complete bipartite $r$-uniform hypergraphs represent networks divided into two regions, while hub-based designs can be modeled by $r$-uniform star hypergraphs, in which a central sensor coordinates the others. Complete $r$-partite $r$-uniform hypergraphs are suitable for representing heterogeneous regions. Linear and cyclic infrastructures can be modeled by $r$-uniform paths and cycles, respectively, illustrating how monitoring can be maintained through minimal coalitions.\\
\indent By focusing on these special hypergraph families, transversal coalitions can be explicitly characterized, providing practical insights for sensor placement that maximize coverage while minimizing cost. This highlights the value of studying structured cases alongside broader theoretical results.\\
\indent In this paper, we determine the transversal coalition number of certain special $r$-uniform hypergraphs.
\begin{definition}\cite{Ber} ({\bf Transversal of a hypergraph})\\
\indent A transversal $T$ of a hypergraph ${H}= (U, \mathcal{E})$ is defined as a subset of the vertex set $U$ such that for every edge $e$ in the set of edges $\mathcal{E}$, the intersection of $T$ and $e$ is non-empty (i.e., $ T \cap e \neq \emptyset $).\\
\indent A transversal in a hypergraph ${H}$ exists exactly when none of the hyperedges are empty. In graph theory, this concept is often referred to as a vertex cover. The transversal number $\boldsymbol\tau({H})$ is defined as the minimum number of vertices in a transversal of ${H}$.
\end{definition}
\begin{definition}\cite{Hen1} ({\bf Transversal Coalition})\\
\indent A transversal coalition in a hypergraph ${H} = (U, \mathcal{E})$ is a pair of disjoint subsets of vertices $U_1$ and $U_2$ such that neither $U_1$ nor $U_2$ is a transversal of ${H}$, but their union $ U_1 \cup U_2$ is a transversal. That is,\\
i) $U_1 \cap e = \emptyset$ for some $e \in \mathcal{E}$,\\
ii) $U_2 \cap e' = \emptyset$ for some $ e' \in \mathcal{E}$, but\\
iii) $ (U_1 \cup U_2) \cap e \neq \emptyset$ for all $e \in \mathcal{E}$.\\
\indent In this case, $U_1$ and $U_2$ are called transversal coalition partners.
\end{definition}
\begin{definition}\cite{Hen1} ({\bf $trc$-partition})\\
\indent A transversal coalition partition ($trc$-partition) of a hypergraph $H$ with vertex set $U$ ($|U|=n$) is a partition of $U$ into subsets $\varphi=\{U_1, U_2, \dots, U_k\}$. Each subset $U_i$ is either a singleton transversal set or, if not, forms a transversal coalition with another non-transversal subset $U_j$. The {\bf transversal coalition number} ${\boldsymbol C_{\tau}(H)}$ is the maximum order $k$ for which such a partition exists. 
\end{definition}
\indent A few examples are presented to illustrate this concept. 
Let ${H}=(U,\mathcal{E})$ be a $3$-uniform hypergraph 
with the vertex set $U=\{u_1, u_2, u_3, u_4, u_5\}$
and hyperedge family $\mathcal{E}=\{ e_1=\{u_1, u_2, u_3\},\ 
~e_2=\{u_2, u_4, u_5\},\ ~e_3=\{u_1, u_3, u_5\},\\\ ~e_4=\{u_3, u_4, u_5\}\}$. A path $P_3$ of length $3$ in ${H}$ is given by $u_1 e_3 u_5 e_2 u_2 e_1 u_3.$ The partition $\varphi=\{\{u_{1}\}, \{u_{2}\}, \{u_{3}\}, \{u_{4}\}, \{u_{5}\}\}$ is a $trc$-partition for the path $P_3$. No set of $\varphi$ is a transversal set, but $\{u_{1}\}$ and $\{u_{4}\}$ forms a transversal coalition; $\{u_{2}\}$ and $\{u_{5}\}$ forms a transversal coalition; $\{u_{3}\}$ and $\{u_{5}\}$ forms a transversal coalition. It follows that every set can be paired with at least one other set to form a transversal coalition. Therefore, $C_\tau(P_3)$=$5$.\\
\indent Now consider a cycle $C_3$ of length 3 in a 3-uniform hypergraph ${H}=(U,\mathcal{E})$ with vertex set $U=\{u_1, u_2, u_3, u_4\}$ and edge set $\mathcal{E}=\{e_1=\{u_{1},u_{2},u_{3}\}, ~e_2=\{u_{1},u_{2},u_{4}\}, ~e_3=\{u_{1},u_{3},u_{4}\}\}$ as $ u_1 e_2 u_4 e_3 u_3 e_1 u_1$. The singleton partition $\varphi_{1}$ of the above cycle is $\{\{u_{1}\}, \{u_{2}\}, \{u_{3}\}, \{u_{4}\}\}$. In $\varphi_1$, the set $\{u_{1}\}$ is a singleton transversal set, while the remaining sets form a transversal coalition that includes at least one set in $\varphi_1$. Here, $\{u_{2}\}$ forms a transversal coalition with $\{u_{3}\}$, and $\{u_{4}\}$. Hence, $C_\tau(C_3)=4$.\\ 
\indent A self-complementary 3-uniform hypergraph $H$ with vertex set
$U=\{u_1,u_2,u_3,u_4,u_5\}$ and edge set
$\mathcal{E}=\{\{u_1,u_2,u_5\},\{u_1,u_4,u_5\},\{u_2,u_3,u_5\},\\\{u_1,u_3,u_4\},\{u_2,u_3,u_4\}\}$. The $trc$-partition is given by
$\varphi_2=\{\{u_1\},\{u_2\},\{u_3\},\\\{u_4\},\{u_5\}\}$. 
Here, each member of $\varphi$ forms a transversal coalition with at least one member. Therefore, $C_\tau(H)=5$.\\
\indent In \cite{Hen1} Henning proved the following result.  
\begin{theorem}
[Henning] A hypergraph $H$ has a transversal coalition partition if and only if there is no edge $e$ in $H$ with $e \neq V (H)$ that is a subset of all edges in $H$.
\end{theorem}
\begin{corollary}[Henning] Every hypergraph $H$ with $\tau(H) \geq 2$ has a transversal coalition partition.
\end{corollary}
\begin{remark} \label{rem1}If $H $ is a hypergraph of order $n$ then $1\leq C_\tau(H)\leq n $ i.e. for nontrivial $H$, $C_\tau(H) > 1$.
\end{remark}
\indent Here, the trivial hypergraph $H$, which consists of only one vertex, is the only hypergraph that meets the lower bound established by Remark \ref{rem1}. In contrast, an $n$-uniform complete hypergraph with $n$ vertices achieves the upper bound. This is because $K_n^n $ contains a single edge that includes all vertices. Consequently, in the $trc$-partition, all elements will be singletons. For the simplest nontrivial $2$-uniform hypergraph with a single edge connecting two vertices, $C_\tau(H) > 1$. Moreover, the transversal coalition number is $2$. \\
\indent For example, consider a complete 4-uniform hypergraph with vertex set $ U=\{u_1, u_2, u_3, u_4\}$ and edge set $ \mathcal{E}=\{\{u_1, u_2, u_3, u_4\}\}$. In this case, all singletons form a singleton transversal set.
Thus, we have $ C_\tau(K_4^4)=4 $.\\
\indent In graph theory, the complete graph $K_n$ admits the singleton vertex partition as its transversal coalition partition. Therefore, each part of the partition consists of exactly one vertex. Hence, $C_\tau(K_n)=n$.\\
\indent In the next theorem, we give the transversal coalition number of $K_n^{r}$.
\begin{theorem}\label{th1}
If $H$ is a complete $r$-uniform hypergraph $K_n^{r}$ then \\i)\hspace{2pt} $C_\tau(K_n^{r})=n$, if $n=r$ or $n=r+1$.\\ 
ii)\hspace{2pt} $C_\tau(K_n^{r})=r+1$ if $r\geq 3$ and $n > r+1$.
\end{theorem}
\begin{proof}
Let $H$ be a complete $r$-uniform hypergraph $K_n^{r}$ with vertex set $U=\{u_1, u_2,\ldots, u_n\}$.\\
i) If $ n = r$, then we obtain a single edge for the hypergraph that contains all the vertices. The $trc$-partition will consist of all singletons. Therefore, $\varphi=\{ \{ u_1 \}, \{ u_2 \},\ldots, \{ u_n \} \}$. Thus, $C_\tau(K_n^{r}) = n = r$.\\
\indent For $n=r+1$, the $r$-uniform hypergraph will have $n$ edges, where the degree of each vertex $u$ is $r=(n-1)$, and it is not present in exactly one edge of an $r$-uniform hypergraph $K_n^{r}$. The vertex $u$ that is present in $ \binom{n}{r} -1=r$ edges forms a transversal coalition with every vertex of the remaining edge. Hence, $C_\tau(K_n^{r})=n$. \\
ii) Now if $r\geq 3$ and $n > r+1$. The complete $r$-uniform hypergraph $K_n^r$ has $\binom{n}{r}$ edges in total. Among these, consider the edge $e = \{u_{n-r+1}, u_{n-r+2}, \ldots, u_n\}$, which consists of exactly $r$ vertices. The remaining $(n-r)$ vertices of $K_n^r$, namely $u_1, u_2, \ldots, u_{n-r}$, do not appear in $e$, but all the remaining edges contain at least one vertex from these $(n-r)$ vertices. Let $\varphi = \{P_1, P_2, P_3, \ldots, P_k\}$ be a $trc$-partition of $K_n^r$. To achieve maximum order $k$, observe that at least one partition set, say $P_1$, must contain $u_1, u_2, \ldots, u_{n-r}$. Moreover,
$P_1 = \{u_1, u_2, \ldots, u_{n-r}\}$. Since $P_1$ contains $(n-r)$ vertices, the remaining $r$ vertices can be partitioned into singletons. Thus $k=r+1$. Let $P_1=\{u_1, u_2,\ldots, u_{n-r}\},~P_2=\{u_{n-r+1}\},~P_3=\{u_{n-r+2}\},\ldots,~ P_{r+1}=\{u_{n}\}$. Here, $P_2, P_3,\ldots,P_{r+1}$ forms a transversal coalition with $P_1$. We now prove that this is the maximum-order $trc$-partition. \\   
\indent Suppose that we partition $P_1$ as $P_{1,1}$ and $P_{1,2}$, which forms a transversal coalition, then we get a new $trc$-partition as  $\varphi'=\{P_{1,1}, P_{1,2}, P_{2}, P_{3},\ldots, P_{r+1}\}$. It suffices to consider $ u_n\in P_{r+1}, u_{n-1}\in P_{r}, u_{n-2}\in P_{r-1},\ldots, u_{n-r+2}\in P_{3}, u_{n-r+1}\in P_{2}$. If $P_{1,1}$ forms a transversal coalition with $P_{1,2}$, then there exists an edge $\{u_{n-r+1},u_{n-r+2},\ldots, u_{n-2},u_{n-1},u_{n}\}$ whose intersection with $P_{1,1}\cup P_{1,2}$ is empty. A contradiction to $ \varphi'$ is a $trc$-partition. Hence, we have $\varphi$ as the only $trc$-partition for $K_n^{r}$ with maximum order $(r+1)$.
Therefore, $C_\tau(K_n^{r})=r+1$. 
\end{proof}
\indent The theorem can be illustrated by the following example.
\begin{example}
Consider a complete $3$-uniform hypergraph $K_5^{3}$ with a vertex set as $ U=\{u_1, u_2, u_3, u_4, u_5\}$ and an edge set $\mathcal{E}=\{\{u_1,u_2,u_3\}, \{u_1,u_2,u_4\},\\ \{u_1,u_2,u_5\}, 
\{u_1,u_3,u_4\}, \{u_1,u_3,u_5\},
\{u_1,u_4,u_5\}, \{u_2,u_3,u_4\}, \{u_2,u_3,u_5\},
\\\{u_2,u_4,u_5\}, \{u_3,u_4,u_5\}\}$. The transversal coalition partition is $P_1=\{u_1, u_2\},\\~P_2=\{u_3\},~P_3=\{u_4\},~ P_4=\{u_5\}$. Here, $P_2, P_3$ and $P_{4}$ forms a transversal coalition with $P_1$. Therefore, $C_\tau(K_5^{3})=4$.
\end{example}
\begin{definition}\cite{Bre}{\bf(Complete Bipartite $r$-uniform hypergraph)}\\
\indent An $r$-uniform hypergraph $H$ with vertex set $U=U_1\cup U_2, ~U_1 \cap U_2=\phi$ and edge set $\mathcal{E}=\{ e : e \subset U, ~|e| = r \text{ and } e \cap U_i \neq \phi, \text{ for } i=1,2 \}$. It is called a complete bipartite $r$-uniform hypergraph. It is denoted as $K^r(U_1,U_2) \text{ or } K^{r}_{m,n} \text{ where } |U_1|=m \text{ and } |U_2|=n.$ If $m=1$ or $n=1$, the hypergraph is referred to as a star. In this case, there exists a vertex $ u \in U $, called the center, such that every hyperedge $ e \in \mathcal{E} $ contains  $u$. That is, $\forall~ e \in \mathcal{E},~ u\in e.$\\
\indent To determine the transversal coalition number of $K^{r}_{m,n}$, we begin with the transversal coalition number of a star.\\
\indent In the following result, we prove that $C_\tau(K_{1,n}^{r})=r$ or $r+1$.
\end{definition}
\begin{theorem}\label{th2} Let $H$ be an $r$-uniform star hypergraph $K_{1,n}^{r}$ then \\
i)\hspace{2pt}  $C_\tau(K_{1,n}^{r})=r$ if $n=r-1$\\
ii)\hspace{2pt}  $C_\tau(K_{1,n}^{r})=r+1$ if $n\geq r$.
\end{theorem} 
\begin{proof}
i) If $n=r-1$, then we get a single edge for the $r$-uniform star hypergraph $K_{1,n}^{r}$. Hence, every vertex is a singleton transversal set, including the center of the star hypergraph. Therefore, $C_\tau(K_{1,n}^{r})=r$, if $n=r-1$. \\
ii) Suppose $n = r$. We can consider the vertex set of $K_{1,n}^{r}$ as $U = U_1 \cup U_2$, where $U_1 = \{u\}$ and $U_2 = \{v_1, v_2, v_3, \ldots, v_n\}$. There are $n$ edges in $K_{1,n}^{r}$, and each vertex of $U_2$ has degree $(n-1)$. Here, $\{u\}$ is a singleton transversal set. Let a ${trc}$-partition be denoted as $\varphi = \{P_1, P_2, P_3, \ldots, P_k\}$.\\
\indent Since every edge of $K_{1,n}^{r}$ contains $u \in U_1$ and the remaining $(r - 1)$ vertices from $U_2$, any vertex in $U_2$ will form a transversal coalition with all other vertices. As a result, there will be $r$ singletons in the partition of $U_2$, leading to $k = r + 1$.\\
\indent We can denote the partitions as follows: $P_1=\{u\}$, $P_2=\{v_1\}$, $P_3=\{v_2\}$, $P_4=\{v_3\}$, and so on, up to $P_r=\{v_{n-1}\}$ and $P_{r+1}=\{v_n\}$. Here, $P_1$ is the singleton transversal set, while the remaining members of the $trc$-partition will form transversal coalitions with one another. Clearly, this represents the maximum order of the $trc$-partition. Therefore, $C_\tau(K_{1,n}^{r})=r+1$. \\ 
\indent If $n > r$, consider the vertex set of $K_{1,n}^{r}$ as $U = U_1 \cup U_2$ partitioned as
$U_1 = \{u\}$ and $U_2 = \{v_1, v_2, v_3, \dots, v_n\}$. There are $\binom{n}{r-1}$ edges in $K_{1,n}^{r}$. Consider the edge $e=\{u, v_1,v_2,\ldots,v_{r-1}\}$. The remaining $\binom{n}{r-1} - 1$ edges contain $(n-r+1)$ vertices of $U_2$, namely $v_r, v_{r+1},v_{r+2},\ldots, v_n$.\\
\indent Let the ${trc}$-partition be $\varphi = \{P_1, P_2, P_3, \dots, P_k\}$. Consider $P_1=\{u\}$, a singleton transversal set. To achieve maximum order, one of the partition sets, say $P_i$, must contain $(n-r+1)$ vertices (with $i \neq 1$). Since $P_i$ contains these $(n-r+1)$ vertices, the remaining $(r-1)$ vertices of $U_2$ can each be placed into singleton partition sets. Consider $P_1=\{u\}, ~P_2=\{v_1, v_2, v_3,\ldots, v_{n-r+1}\}, ~P_3=\{v_{n-r+2}\}, ~P_4=\{v_{n-r+3}\},\ldots, P_r=\{v_{n-1}\}$ and $P_{r+1}=\{v_{n}\}$. Here, $P_2$ forms a transversal coalition with $ P_3,\ldots, P_r, P_{r+1}$.  Hence, $k = r+1$. Therefore, $C_\tau(K_{1,n}^{r})=r+1$.\\ 
\indent Suppose that the set $P_2$ is partitioned into $P_{2,1}$ and $P_{2,2}$, where $P_{2,1}=\{v_1\}$ and $P_{2,2}=\{v_2,v_3,\ldots,v_{n-r+1}\}$. Define $\varphi'=\{P_1, P_{2,1}, P_{2,2}, P_3, \ldots, P_r, P_{r+1}\}$ as the new $trc$-partition. Assume that $P_{2,2}$ forms a transversal coalition with $P_3$. But then there exists an edge $e=\{u, v_1, v_{n-r+3},\ldots, v_{n-1}, v_n\}$ whose intersection with $P_{2,2} \cup P_3$ is empty, a contradiction to $P_{2,2}$ and $P_3$ being transversal coalition partners. In addition, we can observe that $P_{2,2}$ will not form a transversal coalition with any other set in $\varphi'$. Hence, $\varphi$ is the only $trc$-partition of $K_{1,n}^{r}$ with maximum order $(r+1)$. Therefore, $ C_\tau(K_{1,n}^{r})=r+1$.
\end{proof}
\indent The following example serves to illustrate the above theorem.
\begin{example}
Consider a $4$-uniform star hypergraph $K_{1,5}^{4}$ with a vertex set as $U = U_1 \cup U_2$ where $U_1 = \{u\}$ and $U_2 = \{v_1, v_2, v_3, v_4, v_5\}$. The edge set $\mathcal{E}= \{u,v_1,v_2,v_3\},
\{u,v_1,v_2,v_4\}, \{u,v_1,v_2,v_5\}, \{u,v_1,v_3,v_4\},
\{u,v_1,v_3,v_5\},
\{u,v_1,\\v_4,v_5\}, \{u,v_2,v_3,v_4\}, \{u,v_2,v_3,v_5\},
\{u,v_2,v_4,v_5\}, \{u,v_3,v_4,v_5\}\}$. The partition of the transversal coalition is $P_1=\{u\},~P_2=\{v_1, v_2\},~P_3=\{v_3\},~ P_4=\{v_4\},~ P_5=\{v_5\}$, where $P_3, P_4$ and $P_{5}$ form a transversal coalition with $P_2$. Therefore, $C_\tau(K_{1,5}^{4})=5$.   
\end{example}
\indent The following theorem provides the $trc$-number of $K_{m,n}^{r}$, subject to specific conditions on $m$ and $n$.
\begin{theorem}\label{th3} For a complete bipartite $r$-uniform hypergraph $K_{m,n}^{r},r\geq3$ and $m,n >1$,\\ 
i) $C_\tau(K_{m,n}^{r})=r$, if $m+n = r$.\\
ii) $C_\tau(K_{m,n}^{r})=r+1$, if $m+n= r+1$\\
iii) $C_\tau(K_{m,n}^{r})=r+1$, if $m+n > r+1$ and either $m=r-1$ or $n=r-1$.
\end{theorem}
\begin{proof}
Consider a complete bipartite $r$-uniform hypergraph $K_{m,n}^{r}$ with a vertex set $U = U_1 \cup U_2$ where $U_1=\{ u_1, u_2,u_3,\ldots,u_m\}$ and $U_2=\{ v_1, v_2, v_3,\ldots,\\ v_n\}$.
i) If $m+n = r$, then we have a single edge containing $r$ vertices. Each vertex constitutes a singleton transversal set. Hence, $C_\tau(K_{m,n}^{r})=r$.\\
ii) For the case where $m+n = r+1$, we analyze a transversal coalition partition of the $r$-uniform hypergraph $K_{m,n}^{r}$, which we represent as $\varphi=\{P_{1}, P_{2}, P_{3},\ldots, P_{k}\}$. In this hypergraph, the total number of edges is $r+1$, and each vertex has a degree of $r$. Therefore, there exists a single edge $e_i=\{u_2, u_3,\ldots, u_m, v_1, v_2,\ldots, v_n\}$ that does not include $u_1$. Consequently, the singleton set $\{u_1\}$ can form a transversal coalition with each vertex present on the edge $e_i$. It is important to note that in the edge set of $K_{m,n}^{r}$, there is exactly one edge from which a vertex is absent. Therefore, each vertex can form a transversal coalition with all the other vertices in $K_{m,n}^{r}$. This leads us to conclude that the transversal coalition partition can be expressed as $ \varphi = \{\{u_1\}, \{u_2\}, \{u_3\},\ldots, \{u_m\}, \{v_1\}, \{v_2\},\ldots, \{v_n\}\} $. In this structure, every member of $ \varphi$ forms a transversal coalition with all the remaining members of $ \varphi$. Hence, $C_\tau(K_{m,n}^{r})=r+1$.\\
iii) For $ m+n > r+1$, consider a transversal coalition partition of $K_{m,n}^{r}$ as $ \varphi=\{P_{1}, P_{2}, P_{3},\ldots, P_{k}\}$. Suppose $n = r-1 $, then there is an edge $e=\{u_m, v_1, v_2,\ldots, v_n\}$ that includes only one vertex from $U_1$ and the remaining $(r-1)$ vertices from $U_2$. The other edges consist of vertices from $\{u_1, u_2, u_3,\ldots, u_{m-1}\}$. Clearly, we can partition $U_1$ as $\{u_1, u_2, u_3,\ldots, u_{m-1}\} \cup \{u_m\}$ and $U_2$ as singletons. Moreover, $\{u_1, u_2, u_3,\ldots, u_{m-1}\}$ will form a transversal coalition with each vertex of $ U_2$ and $\{u_m\}$. This allows us to form a $trc$-partition represented as $\varphi=\{ P_1=\{u_1, u_2, u_3,\ldots, u_{m-1}\}, ~P_2=\{u_m\}, ~P_3=\{v_1\}, ~P_4=\{v_2\},\ldots, ~P_{r+1}=\{v_n\} \}$. Thus, $k=r + 1$. Hence, $C_\tau(K_{m,n}^{r})=r + 1$. Similarly, we can prove that $ C_\tau(K_{m,n}^{r}) = r + 1$ if $m=r-1$.\\
\indent Here, $P_1$ and $P_2$ are transversal coalition partners, and $P_3, P_4, \ldots, P_{r}, P_{r+1}$ form a transversal coalition with $P_{1}$. Suppose that we partition $P_1$ as $P_{1,1}$ and $P_{1,2}$, where $P_{1,1}=\{u_1, u_2, u_3, u_4, \ldots, u_{m-2}\}$ and $P_{1,2}=\{u_{m-1}\}$. We then obtain a new $trc$-partition $\varphi'=\{P_{1,1}, P_{1,2}, P_2, P_3,\ldots, P_{r+1}\}$. Clearly $P_{1,1}$ does not form a transversal coalition with $P_{1,2}$ since there is an  edge $\{u_m, v_1, v_2,\ldots, v_{r-1}\}$ that does not intersect $P_{1,1}$ or $P_{1,2}$. In addition, $P_{1,1}$ does not form a transversal coalition with any other set in $\varphi'$. Hence $(r+1)$ is the maximum order of the $trc$-partition. This implies $C_\tau(K_{m,n}^r) = r + 1$.
\end{proof}
\indent The following example provides an illustration of the theorem for the case $m=r-1$.
\begin{example}
Consider a complete bipartite $3$-uniform hypergraph $K_{2,4}^{3}$ with a vertex set as $U = U_1 \cup U_2$ where $U_1 = \{u_1, u_2\}$ and $U_2 = \{v_1, v_2, v_3, v_4\}$. The edge set $\mathcal{E}= \{u_1,v_1,v_2\},
\{u_1,v_1,v_3\}, \{u_1,v_1,v_4\}, \{u_1,v_2,v_3\},
\{u_1,v_2,v_4\},\\ \{u_1,v_3,v_4\},
\{u_2,v_1,v_2\}, \{u_2,v_1,v_3\}, \{u_2,v_1,v_4\},
\{u_2,v_2,v_3\}, \{u_2,v_2,v_4\},\\ \{u_2,v_3,v_4\},
\{u_1,u_2,v_1\}, \{u_1,u_2,v_2\}, \{u_1,u_2,v_3\},
\{u_1,u_2,v_4\} \}$. The $trc$-partition is $\varphi=\{ P_1=\{u_1\}, ~P_2=\{u_2\}, ~P_3=\{v_1, v_2, v_3 \}, ~P_4=\{v_4\}\}$, where $P_3$ and $P_4$ are transversal coalition partners, and $P_1$ and $P_2$ are transversal coalition partners. In addition, they form a transversal coalition with $P_{3}$. Hence, $C_\tau(K_{2,4}^3) =4$.
\end{example}
\begin{definition}\cite{Bre} \textbf{(Complete $r$-partite $r$-uniform hypergraph)}\\
\indent A complete $r$-partite $r$-uniform hypergraph is a hypergraph whose vertex set is partitioned into $r$ pairwise disjoint parts $ U = U_1 \cup U_2 \cup\ldots \cup U_r,$
where $U_i \cap U_j=\varnothing$ for all $i \neq j$. Each hyperedge contains exactly one vertex from each part. Thus, the set of edges is $\mathcal{E}=\bigl\{ \{u_1, u_2,\ldots, u_r\} \mid u_i \in U_i \text{ for } i = 1,2,\ldots,r \bigr\}$. The hypergraph is denoted by $K_r(U_1, U_2,\ldots, U_r)$. If $|U_i|=n_i$ for $i=1,2,\ldots,r$, it is also denoted by $K^{r}_{n_1,n_2,\ldots,n_r}$.
\end{definition}
\indent In the following theorem, we give bounds for the transversal coalition number of $K_{n_1, n_2,n_3,\ldots,n_r}^{r}$. 
\begin{theorem}\label{th4}
Let $K_{n_1, n_2,n_3,\ldots,n_r}^{r}$ be a complete $r$-partite $r$-uniform hypergraph, then $r\leq C_\tau(K_{n_1, n_2,n_3,\ldots,n_r}^{r})\leq2r$.    
\end{theorem}
\begin{proof}
Let $K_{n_1, n_2, n_3,\ldots, n_r}^{r}$ be a complete $r$-partite $r$-uniform hypergraph. Let $U=U_1\cup U_2 \cup\ldots\cup U_r$ be the vertex set where \\
$U_1=\{ u_{1,1}, u_{1,2}, u_{1,3},\ldots, u_{1,n_1}\}$,\\
$ U_2=\{ u_{2,1}, u_{2,2}, u_{2,3},\ldots, u_{2,n_2} \}$,\\
$U_3=\{ u_{3,1}, u_{3,2}, u_{3,3},\ldots, u_{3,n_3} \}$,\\
\vdots \\
$U_r=\{ u_{r,1}, u_{r,2}, u_{r,3},\ldots, u_{r,n_r} \}$ with $|U_i| = n_i, ~i=1,2,3,\ldots,r$.\\
\indent Let $\varphi=\{P_1, P_2,\ldots,P_k\}$ be the $trc$-partition of $K_{n_1, n_2,n_3,\ldots,n_r}^{r}$. If $\left| U_i \right| = 1 $, $\forall~ i=1, 2,\dots, r$, then we get a single edge containing $r$ vertices. Here, every vertex is a singleton transversal set. Therefore, $ C_\tau(K_{n_1, n_2, n_3,\dots, n_r}^{r})=r.$ \\
\indent If $|U_i|=1$ for at least one $i$, and $|U_j| \geq 2$ for each $j$ with $j=1,2,\ldots, r$ and $i \neq j$, then $U_i$ serves as a singleton transversal set. Additionally, we can partition each $U_j$ into two subsets, $U_{j,1}$ and $U_{j,2}$, such that $U_{j,1} \cup U_{j,2}$ forms a transversal set for $1 \leq j \leq r, ~j\neq i$. This leads to the conclusion that $ k=1 + 2(r - 1) = 2r - 1$. Hence, $C(K_{n_1, n_2, n_3,\ldots, n_r}^{r})=2r - 1$. \\
\indent If  $ \left| U_i \right| \geq 2$ for each $i=1,2,\ldots, r$, then we need to partition each $ U_i $ into two subsets $U_{i,1}$ and $U_{i,2}$ where $U_{i,1}\cup U_{i,2}$, $1\leq i\leq r$ is a transversal set. Let $\varphi=\{P_1=U_{1,1}, P_2=U_{1,2}, P_3=U_{2,1}, P_4=U_{2,2},\ldots, P_{2r-1}=U_{r,1}, P_{2r}=U_{r,2}\}$ where $U_{1,1}=\{u_{1,1}, u_{1,3}, u_{1,5},\ldots, u_{1,r}\}$, $U_{1,2}=\{u_{1,2}, u_{1,4}, u_{1,6},\ldots, u_{1,r-1}\},\ldots, \\ U_{r,1}=\{u_{r,1}, u_{r,3}, u_{r,5},\ldots, u_{r,r}\}$ and $U_{r,2}=\{u_{r,2}, u_{r,4}, u_{r,6},\ldots, u_{r,r-1}\}$. Here, $U_{1,1}$ forms a transversal coalition with $U_{1,2}$; $U_{2,1}$ forms a transversal coalition with $U_{2,2};\ldots;U_{r,1}$ forms a transversal coalition with $U_{r,2}.$ Therefore, $\varphi$ is a $trc$-partition and $k=2r$. Thus, $ C_\tau(K_{n_1, n_2, n_3,\ldots, n_r}^{r})=2r$.\\
\indent Suppose that we partition $U_{1,1}$ as $W_{1,1}\cup W_{1,2}$, where $W_{1,1}=\{u_{1,1}, u_{1,5},\ldots,\\ u_{1,r-2}\}$ and $W_{1,2}=\{u_{1,3}, u_{1,7},\ldots, u_{1,r}\}$. Let $\varphi'=\{W_{1,1}, W_{1,2}, U_{1,2} , U_{2,1}, U_{2,2},\\ \ldots, U_{r,1}, U_{r,2} \}$. There exists an edge $\{u_{1,2}, u_{2,2}, u_{3,2},\ldots, u_{r,2}\}$, that does not intersect $W_{1,1}$ or $W_{1,2}$. In addition, neither $W_{1,1}$ nor $W_{1,2}$ forms a transversal coalition with any other set in $\varphi'$. Therefore, $\varphi'$ is not a $trc$-partition. Hence $\varphi$ is the only $trc$-partition with maximum order. Therefore, $ C(K_{n_1, n_2, n_3, \dots, n_r}^{r})=2r.$  
\end{proof}
\indent To illustrate the theorem, consider the following example.
\begin{example}
Let $K_{n_1, n_2, n_3, n_4}^{4}$ be a complete $4$-partite $4$-uniform hypergraph with vertex set as $ U=U_1\cup U_2 \cup U_3\cup U_4$ where 
$U_1=\{ u_{1,1}, u_{1,2}\}, ~U_2=\{ u_{2,1}, u_{2,2}\}, ~U_3=\{ u_{3,1}, u_{3,2}\}, ~U_4=\{ u_{4,1}, u_{4,2}\}$ with $|U_i| = n_i=2, ~i=1,2,\\3,4$. The edge set $\mathcal{E}=\{ \{ u_{1,1}, u_{2,1}, u_{3,1}, u_{4,1}\}, \{u_{1,1}, u_{2,1}, u_{3,1}, u_{4,2}\}, \{u_{1,1}, u_{2,1},\\ u_{3,2}, u_{4,1}\}, \{u_{1,1}, u_{2,1}, u_{3,2}, u_{4,2}\}, \{u_{1,1}, u_{2,2}, u_{3,1}, u_{4,1}\}, \{u_{1,1}, u_{2,2}, u_{3,1}, u_{4,2}\},\\ \{u_{1,1}, u_{2,2}, u_{3,2}, u_{4,1}\}, \{u_{1,1}, u_{2,2}, u_{3,2}, u_{4,2}\}, \{u_{1,2}, u_{2,1}, u_{3,1}, u_{4,1}\}, \{u_{1,2}, u_{2,1}, u_{3,1},\\ u_{4,2}\}, \{u_{1,2}, u_{2,1}, u_{3,2}, u_{4,1}\}, \{u_{1,2}, u_{2,1}, u_{3,2}, u_{4,2}\}, \{u_{1,2}, u_{2,2}, u_{3,1}, u_{4,1}\}, \{u_{1,2}, u_{2,2},\\ u_{3,1}, u_{4,2}\}, \{u_{1,2}, u_{2,2}, u_{3,2}, u_{4,1}\}, \{u_{1,2}, u_{2,2}, u_{3,2}, u_{4,2}\}\}$. The $trc$-partition is $\varphi= \{P_1=\{u_{1,1}\}, P_2=\{u_{1,2}\}, P_3=\{u_{2,1}\}, P_4=\{u_{2,2}\}, P_5=\{u_{3,1}\}, P_6=\{u_{3,2}\}, P_7=\{u_{4,1}\}, P_8=\{u_{4,2}\} \}$. Here, $P_1$ forms a transversal coalition with $P_2$; $P_3$ forms a transversal coalition with $P_4$; $P_5$ forms a transversal coalition with $P_6$; $P_7$ forms a transversal coalition with $P_8$. Therefore, $ C(K_{n_1, n_2, n_3, n_4}^{4})=2(4)=8.$   
\end{example}
\section{Transversal Coalitions in Paths and Cycles}
\hspace{\parindent} Now we will determine the transversal coalition numbers of all linear paths and linear cycles of an $r$-uniform hypergraph.
\begin{definition}\cite{Bre} Let $H=(U, \mathcal{E})$ be a hypergraph without isolated vertices. A path $P$ in $H$ from a vertex $x$ to a vertex $y$ is a sequence of the form
$x=x_1, e_1, x_2, e_2,\ldots, x_s, e_s, x_{s+1}=y$
in which $ x_1, x_2,\ldots, x_s, x_{s+1}$ are distinct vertices, and $ e_1, e_2,\ldots, e_s$ are distinct hyperedges, such that each consecutive pair of vertices $ x_i$ and  $x_{i+1}$ belongs to the corresponding hyperedge $e_i$ for all $ i=1, 2,\ldots, s$. If $x = x_1 = x_{s+1} = y$, the path is called a cycle. The integer $s$ is the length of the path $P$.
\end{definition}
\begin{definition}
A linear path in an $r$-uniform hypergraph $ H = (U, \mathcal{E})$ is a sequence of distinct hyperedges $P = (e_1, e_2,\ldots, e_k)$
such that for every $i$ with $ 1 \leq i < k$, the intersection of consecutive hyperedges $e_i$ and $e_{i+1}$ contains exactly one vertex, that is, $|e_i \cap e_{i+1}| = 1$, and for any pair of hyperedges $ e_i$ and $e_j$ with indices satisfying $|i - j| > 1$, the intersection is empty, meaning $ e_i \cap e_j = \emptyset$.
\end{definition}
\begin{definition}
A linear cycle in an $r$-uniform hypergraph $H = (U, \mathcal{E})$ is a sequence of distinct hyperedges $C=(e_1, e_2, \ldots, e_k),$ with $k \geq 2$, such that for every $i$ with $1 \leq i < k$, the intersection $ e_i \cap e_{i+1}$ contains exactly one vertex, that is, $|e_i \cap e_{i+1}| = 1$, the intersection $e_k \cap e_1$ also contains exactly one vertex, ensuring the cycle closes, and for all other pairs of indices $ i, j$ with $ j \neq i \pm 1 \pmod{k}$, the intersection $ e_i \cap e_j$ is empty.   
\end{definition}
\indent The next theorem gives the transversal coalition number of a linear path $P_n$ of length $n$ in an $r$-uniform hypergraph.
\begin{theorem}\label{th5} For a linear path $ P_n$ of length $n$ in an $r$-uniform hypergraph, $C_\tau(P_n)=2r, ~n>2$
\end{theorem}
\begin{proof}
Let $P_n$ be a linear path of length $n$ in an $r$-uniform hypergraph, and let $\varphi=\{U_1,U_2,U_3,\ldots,U_k\}$ be the $trc$-partition of $P_n$. Consider a linear path of length $3$ in an $r$-uniform hypergraph as $P_3=u_1 e_1 u_2 e_2 u_3 e_3 u_4$ with edge set
$\mathcal{E}=\{e_1=\{u_1,u_2,x_{1,1},\ldots,x_{1,r-2}\},
~e_2=\{u_2,u_3,x_{2,1},\ldots,x_{2,r-2}\},
~e_3=\{u_3,u_4,x_{3,1},\ldots,x_{3,r-2}\}\}.$
No single vertex intersects every edge of $P_3$. The partition consisting of each vertex $u_i$ as a singleton and the remaining $(r-2)$ vertices of each edge arranged as singletons and doubletons forms a $trc$-partition of $P_3$ of order $2r$. Specifically, $\varphi=\{\{u_1\},\{u_2\},\{u_3\},\{u_4\},\{x_{1,1},\\x_{3,1}\},\{x_{2,1}\},\{x_{1,2},x_{3,2}\},\{x_{2,2}\},\ldots,\{x_{1,r-2},x_{3,r-2}\},\{x_{2,r-2}\}\}.$
Here, $\{u_1\}$ \\forms a transversal coalition with $\{u_3\}$; $\{u_2\}$ forms a transversal coalition with $\{u_3\}$ and $\{u_4\}$; $\{x_{1,1},x_{3,1}\}$ forms a transversal coalition with $\{u_2\}$, $\{u_3\}$, and $\{x_{2,1}\}$; and, in general, $\{x_{1,r-2},x_{3,r-2}\}$ forms a transversal coalition with $\{x_{2,r-2}\}$. Hence, the transversal coalition number is $C_\tau(P_3)=2(r-2)+4=2r.$\\
 \indent Consider a linear path of length $4$ in an $r$-uniform hypergraph as $P_4=u_1 e_1 u_2 e_2 u_3 e_3 u_4 e_4 u_5$, with edge set
 $\mathcal{E}=\{e_1 \mathord{=} \{u_1, u_2, x_{1,1},\ldots, x_{1,r-2}\},
 ~e_2=\{u_2, u_3, x_{2,1},\ldots, x_{2,r-2}\},
 ~e_3=\{u_3, u_4,  x_{3,1},\ldots, x_{3,r-2}\},
 ~e_4=\{u_4, u_5, x_{4,1},\ldots,\\ x_{4,r-2}\}\}$. If the partition of $P_4$ includes each vertex $u_i$ as a singleton, then there will be at least one edge that does not intersect the union of two singletons $\{u_i\}$ and $\{u_j\}$, where $ i \ne j$. To ensure a non-empty intersection with the edges of $P_4$, we must merge a fifth vertex with one of the singleton sets. Since the path length is even, the remaining $(r - 2)$ vertices of each edge in $P_4$ must be grouped as doubletons in partition $\varphi$, for $ \varphi$ to qualify as a $trc$-partition. The $trc$-partition is $\varphi=\{\{u_{1},u_{5}\}, \{u_{2}\}, \{u_{3}\}, \{u_{4}\}, \{x_{1,1}, x_{3,1}\}, \{x_{2,1}, x_{4,1}\},\\\ldots,\{x_{1,r-2}, x_{3,r-2}\}, \{x_{2,r-2}, x_{4,r-2}\} \}$. Here, $\{u_{1},u_{5}\}$ forms a transversal coalition with $\{u_{3}\}$; $\{u_{2}\}$ forms a transversal coalition with $\{u_{4}\}$; $\{x_{1,1},x_{3,1}\}$ forms a transversal coalition with $\{x_{2,1}, x_{4,1}\};\ldots;\{x_{1,r-2}, x_{3,r-2}\}$ forms a transversal coalition with $\{x_{2,r-2}, x_{4,r-2}\}$. The transversal coalition number $C_\tau(P_4)=2(r-2)+4 = 2r$.\\
\indent  Now, consider a linear path of length $5$ in an $r$-uniform hypergraph as $P_5=u_1 e_1 u_2e_2u_3e_3u_4e_4u_5e_5u_6$, with edge set 
$\mathcal{E}=\{e_1=\{u_1,u_2,x_{1,1}\!,\ldots\!,x_{1,r-2}\},
\\~e_2=\{u_2,u_3,x_{2,1}\!,\ldots\!,x_{2,r-2}\},
~e_3=\{u_3,u_4,x_{3,1}\!,\ldots\!,x_{3,r-2}\},
~e_4=\{u_4,u_5,x_{4,1}\!,\\\ldots\!,x_{4,r-2}\},
~e_5=\{u_5,u_6,x_{5,1}\!,\ldots\!,x_{5,r-2}\}\}$. For $\varphi$ to be a valid $trc$-partition, placing all vertices $u_i$ as singletons yields at least one edge that does not intersect the union of any two singletons $\{u_i\}$ and $\{u_j\}$, where $i \ne j$. To achieve a non-empty intersection with the edges of $P_5$, the fifth and sixth vertices must be merged with adjacent singleton vertices. As the path $P_5$ has length $5$, the remaining $(r-2)$ vertices from each edge must be organized into doubletons and tripletons within the partition $\varphi$, ensuring that $\varphi$ is a $trc$-partition. The $trc$-partition is $\varphi=\{\{u_1, u_5\}, \{u_2, u_6\}, \{u_3\}, \{u_4\},\{x_{1,1}, x_{3,1}, x_{5,1}\}, \{x_{2,1}, x_{4,1}\},\\ \ldots, \{x_{1,r-2}, x_{3,r-2}, x_{5,r-2}\}, \{x_{2,r-2}, x_{4,r-2}\}\}$. Here, $\{u_{1},u_{5}\}$ forms a transversal coalition with $\{u_{3}\}$; $\{u_{2}, u_{6}\}$ forms a transversal coalition with $\{u_{4}\}$; $\{x_{1,1},x_{3,1}, x_{5,1}\}$ forms a transversal coalition with $\{x_{2,1}, x_{4,1}\};\ldots;\{x_{1,r-2},\\x_{3,r-2},x_{5,r-2}\}$ forms a transversal coalition with $\{x_{2,r-2}, x_{4,r-2}\}$. The transversal coalition number $C_\tau(P_5)=2(r-2) + 4 = 2r$.\\
\indent In general, consider a linear path of length $n$ in an $r$-uniform hypergraph as $P_n=u_1 e_1 u_2e_2u_3e_3u_4e_4u_5e_5u_6e_6\ldots e_{n}u_{n+1} $ with edge set $\mathcal{E}=\{e_1=\{u_1, u_2, x_{1,1}\!,\ldots, x_{1,r-2}\}\!,
~ e_2=\{u_2, u_3, x_{2,1},\ldots, x_{2,r-2}\}\!,
~e_3=\{u_3, u_4, x_{3,1},\ldots,\\ x_{3,r-2}\}\!,
~e_4=\{u_4,u_5, x_{4,1},\ldots, x_{4,r-2}\}\!,
~e_5=\{u_5, u_6, x_{5,1},\ldots, x_{5,r-2}\}\!,
\ldots, ~e_{n}=\{u_n, u_{n+1}, x_{n,1},\ldots, x_{n,r-2}\}\!\}$.
Let the $trc$-partition be $\varphi=\{U_1, U_2, U_3, U_4, U_5,\\ U_6,\ldots, U_{2r-1},U_{2r}\}$.\\
{\bf Case 1:} If $n$ is odd, then we can consider 
$U_1=\{u_i|(i-1)\equiv 0 \pmod{4}\}$,\\
$U_2=\{u_i|(i-1)\equiv 1 \pmod{4}\}$,\\
$U_3=\{u_i|(i-1)\equiv 2 \pmod{4}\}$,\\
$U_4=\{u_i|(i-1)\equiv 3 \pmod{4}\}$,\\
$U_5=\{x_{1,1}, x_{3,1},x_{5,1},\ldots, x_{n,1}\}$,\\
$U_6=\{x_{2,1}, x_{4,1},x_{6,1},\ldots, x_{n-1,1}\},\\ \vdots  $\\
$U_{2r-1}=\{x_{1,r-2}, x_{3,r-2},x_{5,r-2},\ldots, x_{n,r-2}\}$,\\
$U_{2r}=\{x_{2,r-2}, x_{4,r-2},x_{6,r-2},\ldots, x_{n-1,r-2}\}.$\\
\indent Here, $U_1$ forms a transversal coalition with $U_3$; $U_2$ forms a transversal coalition with $U_4$; $U_5$ forms a transversal coalition with $U_6$; In a similar manner $U_{2r-1}$ forms a transversal coalition with $U_{2r}$. Therefore, $C_\tau(P_n)=2r$.\\
{\bf Case 2:} If $n$ is even, then we can consider 
$U_1=\{u_i|(i-1)\equiv 0 \pmod{4}\}$,\\
$U_2=\{u_i|(i-1)\equiv 1 \pmod{4}\}$,\\
$U_3=\{u_i|(i-1)\equiv 2 \pmod{4}\}$,\\
$U_4=\{u_i|(i-1)\equiv 3 \pmod{4}\}$,\\
$U_5=\{x_{1,1}, x_{3,1},x_{5,1},\ldots, x_{n-1,1}\},$\\
$U_6=\{x_{2,1}, x_{4,1},x_{6,1},\ldots, x_{n,1}\},\\ \vdots  $\\
$U_{2r-1}=\{x_{1,r-2}, x_{3,r-2},x_{5,r-2},\ldots, x_{n-1,r-2}\}$,\\
$U_{2r}=\{x_{2,r-2}, x_{4,r-2},x_{6,r-2},\ldots, x_{n,r-2}\}.$\\
\indent Here, $U_1$ forms a transversal coalition with $U_3$; $U_2$ forms a transversal coalition with $U_4$; $U_5$ forms a transversal coalition with $U_6$; In a similar manner $U_{2r-1}$ forms a transversal coalition with $U_{2r}$. Therefore, $C_\tau(P_n)=2r$.\\
\indent Suppose that we partition the set $U_5$ as $U_{5,1}\cup U_{5,2}$, where $U_{5,1}=\{x_{1,1}, x_{5,1},\\\ldots, x_{n-1,1}\}$ and $U_{5,2}=\{x_{3,1}, x_{7,1},\ldots, x_{n-3,1}\}$. Let $\varphi'=\{U_1, U_2, U_3, U_4, U_{5,1},\\U_{5,2}, U_6,\ldots, U_{2r-1}, U_{2r}\}$ be the new $trc$-partition. If $U_{5,1}$ forms a transversal coalition with $U_6$ but then there is an edge $~e_3=\{u_3, u_4, x_{3,1},\ldots, x_{3,r-2}\}$ that does not intersect $U_{5,1} \cup U_6$. Observe that $U_{5,1}$ and $U_{5,2}$ do not form a transversal coalition with any other set in $\varphi'$. Furthermore, the set $U_6$ does not have a transversal coalition partner in $\varphi'$. This leads to a contradiction, as $\varphi'$ cannot be a $trc$-partition. More generally, if we construct another partition of maximum order taking the proper subsets of any set in $\varphi$, then the resulting partition will necessarily contain at least one set that has no transversal coalition partner. Hence, $\varphi$ is the maximum-order $trc$-partition. Therefore, $C_\tau(P_n)=2r$.	
\end{proof}
\indent An illustrative example of the theorem is given below.
\begin{example}
Consider a linear path of length $4$ in an $4$-uniform hypergraph $P_4=u_1 e_1 u_2 e_2 u_3 e_3 u_4 e_4 u_5$, with edge set
 $\mathcal{E}=\{e_1 \mathord{=} \{u_1, u_2, x_{1,1}, x_{1,2}\},
 ~e_2=\{u_2, u_3, x_{2,1}, x_{2,2}\},
 ~e_3=\{u_3, u_4,  x_{3,1}, x_{3,2}\},
 ~e_4=\{u_4, u_5, x_{4,1}, x_{4,2}\}\}$. The $trc$-partition is $\varphi=\{P_1=\{u_{1},u_{5}\}, P_2=\{u_{2}\}, P_3=\{u_{3}\}, P_4=\{u_{4}\}, P_5=\{x_{1,1}, x_{3,1}\}, P_6=\{x_{2,1}, x_{4,1}\}, P_7=\{x_{1,2}, x_{3,2}\}, P_8=\{x_{2,2}, x_{4,2}\} \}$. Here, $P_1$ forms a transversal coalition with $P_3$; $P_2$ forms a transversal coalition with $P_4$; $P_5$ forms a transversal coalition with $P_6$; $P_7$ forms a transversal coalition with $P_8$. Therefore, $C_\tau(P_4)=8$.
 \end{example}
\par In the next theorem, we determine the transversal coalition number of a linear cycle $C_n$ of length $n$ in an $r$-uniform hypergraph.
\begin{theorem}\label{th6}
For a linear cycle $C_n$ in an $r$-uniform hypergraph of length $n$ where $n>2$, $C_\tau(C_n)=3(r-1)$, if  $n$-odd and $C_\tau(C_n)=2r$, if $n$-even.
\end{theorem}
\begin{proof}
Let $C_n$ be a linear cycle of length $n$ in an $r$-uniform hypergraph. Let $\varphi$ = \{$P_{1}, P_{2}, P_{3},\ldots,P_{k}$\} be the $trc$-partition of $C_n$. Consider a linear cycle of length $5$ in an $r$-uniform hypergraph as $C_5 = u_1 e_1 u_2e_2u_3e_3u_4e_4u_5e_5u_1$ with edge set
$\mathcal{E}=\{e_1 = \{u_1, u_2, x_{1,1},\ldots, x_{1,r-2}\}, 
~e_2=\{u_2, u_3, x_{2,1},\ldots, x_{2,r-2}\},\\
~ e_3=\{u_3, u_4,  x_{3,1},\ldots, x_{3,r-2}\},
~ e_4=\{u_4, u_5, x_{4,1},\ldots, x_{4,r-2}\},
~ e_5=\{u_5, u_1, x_{5,1},\\\ldots, x_{5,r-2}\}\}$.
In a linear cycle $C_5$, each edge connects two vertices and includes those two vertices together with $(r - 2)$ other distinct vertices. For $\varphi$ to be a valid transversal coalition partition, placing all vertices $u_i$ as singletons results in at least one edge that does not intersect with the union of two singletons $ \{u_i\}$ and $\{u_j\}$, where $i \ne j$. To ensure a non-empty intersection with the edges of $C_5$, two pairs of non-adjacent vertices from $u_i$'s must be merged, while one vertex remains as a singleton. As the linear cycle $C_5$ has length five, the remaining $(r - 2)$ vertices from each edge must be arranged into singletons and tripletons within the partition $\varphi$, ensuring that $\varphi$ satisfies the conditions of a $trc$-partition. Observe that we need to partition the vertices at the $ j^ {{th}}$ position of every edge with $j\neq 1,2$ into at most three sets. And the $u_i$'s can be partitioned into at most $3$ sets. Thus, $k=3(r-2)+3=3(r-1)$.
The resulting $trc$-partition is: 
$\varphi = \{ P_1=\{u_1\}, ~P_2=\{u_2, u_4\}, ~P_3=\{u_3, u_5\}, ~P_4=\{x_{1,1}\}, ~P_5=\{x_{2,1}, x_{3,1}, x_{4,1}\}, ~P_6=\{x_{5,1}\},\ldots, ~P_{3r-5}=\{x_{1,r-2}\}, ~P_{3r-4}=\{x_{2,r-2}, x_{3,r-2}, x_{4,r-2}\}, ~P_{3(r-1)}=\{x_{5,r-2}\}\}$. Here, $P_2$ forms a transversal coalition with $P_1, P_6, P_9,\ldots, P_{3(r-1)}$;  $P_3$ forms a transversal coalition with $P_1, P_4, P_7,\ldots,P_{3r-5}$; In addition, $P_5$, $ P_8, P_{11},\ldots, P_{3r-4}$ forms a transversal coalition with $P_1$.
The transversal coalition number $C_\tau(C_5) = 3(r-1).$\\
\indent Consider a linear cycle of length $6$ in an $r$-uniform hypergraph as
$C_6 = u_1 e_1 u_2e_2 u_3e_3u_4 e_4u_5e_5 u_6e_6u_1$ with edge set 
$\mathcal{E}=\{e_1=\{u_1, u_2, x_{1,1},\ldots, x_{1,r-2}\},
\\~e_2=\{u_2, u_3, x_{2,1},\ldots, x_{2,r-2}\},
~e_3 = \{u_3, u_4, x_{3,1},\ldots, x_{3,r-2}\},
~e_4 = \{u_4, u_5, x_{4,1},\\\ldots, x_{4,r-2}\},
~e_5 = \{u_5, u_6, x_{5,1},\ldots, x_{5,r-2}\},
~e_6=\{u_6, u_1, x_{6,1},\ldots, x_{6,r-2}\}\}$. To ensure that $\varphi$ constitutes a valid $C_\tau(C_6)$ partition, assigning each vertex $u_i$ to its own set of singletons would result in at least one edge that does not intersect with the union of any two singletons $\{u_i\}$ and $\{u_j\}$, where $ i \ne j$. Therefore, at least two non-adjacent vertices from the $u_i$'s must be grouped to guarantee a non-empty intersection with the edges of $C_6$, while two other vertices from the $u_i$'s remain in singleton sets. Since $C_6$ is a cycle of length $6$, the remaining $(r - 2)$ vertices on each edge must be organized into a combination of doubletons and a 4-vertex set within the partition $\varphi$. This configuration ensures that $\varphi$ satisfies the requirements of a $trc$-partition, since the vertices are merged so that every edge has a non-empty intersection. In particular, the vertices at the $j^{{th}}$ position of every edge with $j\neq 1,2$ need to be partitioned into at most two sets, and the $u_i$'s can be partitioned into at most 4 sets. Thus, $k=2(r-2)+4=2r$. The resulting $trc$-partition is: $\varphi=\{P_1=\{u_1\}, ~P_2=\{u_2, u_4\}, ~P_3=\{u_3, u_5\}, ~P_4=\{u_6\}, P_5=\{x_{1,1},x_{6,1} \}, ~P_6=\{x_{2,1},x_{3,1}, x_{4,1}, x_{5,1}\}, \ldots, ~P_{2r-1}=\{x_{1,r-2},x_{6,r-2}\}, ~P_{2r}=\{x_{2,r-2}, x_{3,r-2}, x_{4,r-2}, x_{5,r-2}\}\}$.\\
\indent Here, $P_1$ forms a transversal coalition with $P_3$, $P_6$, $P_8$, $P_{10}$,$\ldots$,$P_{2r}$; $P_4$ forms a transversal coalition with $P_2$; $P_5$, $P_7$, $P_9$ and $P_{11}$ form a transversal coalition with  $P_3$, $P_6$, $P_8$, $P_{10}$,$\ldots$,$P_{2r}$; We can easily check that $P_{2r-1}$ forms a transversal coalition with $P_3$, $P_6$, $P_8$, $P_{10}$,$\ldots$,$P_{2r}$. Here, each member in the $trc$-partition is a transversal coalition partner of exactly $(r-1)$ members.
The transversal coalition number $C_\tau(C_6) = 2r$. \\
\indent In general, we can write a $trc$-partition of a cycle $C_n$ for any $n$. Let $C_n=u_1 e_1 u_2e_2u_3e_3u_4e_4u_5e_5u_6\ldots u_ne_nu_1$ be a cycle of length $n$ with edge set
$\mathcal{E}=\{e_1=\{u_1, u_2, x_{1,1},\ldots, x_{1,r-2}\},
~e_2=\{u_2, u_3, x_{2,1},\ldots, x_{2,r-2}\},
~e_3=\{u_3, u_4, x_{3,1},\ldots, x_{3,r-2}\},
~e_4=\{u_4, u_5, x_{4,1},\ldots, x_{4,r-2}\},
~e_5=\{u_5, u_6, x_{5,1},\ldots,\\ x_{5,r-2}\},\ldots
~e_n=\{u_n, u_1, x_{n,1},\ldots, x_{n,r-2}\}\}$.\\
{\bf Case 1:} If $n$ is even, then we get a $trc$-partition as $\varphi=\{ P_1, ~P_2, ~P_3, ~P_4, ~P_5,\\ ~P_6, ~P_7, ~P_8,\ldots, P_{2r-1}, P_{2r}\}=\{ \{u_1\}, \{u_2,u_4\},\{u_3,u_5,\ldots,u_{n-1}\}, \{u_6,u_8,\ldots, \\u_n\}, \{ x_{1,1},x_{n,1}\}, \{ x_{2,1}, x_{3,1}, x_{4,1},\ldots, x_{n-1,1}\}, \{ x_{1,2}, x_{n,2}\},\{ x_{2,2}, x_{3,2}, x_{4,2},\ldots,\\ x_{n-1,2}\},\ldots, \{ x_{1,r-2}, x_{n,r-2}\}, \{ x_{2,r-2}, x_{3,r-2},\ldots, x_{n-1,r-2}\}\}$. Here, $P_1$ forms a transversal coalition with $P_3$, $P_6$, $P_8$, $P_{10}$,$\ldots$,$P_{2r}$; $P_4$ forms a transversal coalition with $P_2$; $P_5$, $P_7$, $P_9$ and $P_{11}$ form a transversal coalition with  $P_3$, $P_6$, $P_8$, $P_{10}$,$\ldots$,$P_{2r}$. We can easily check that $P_{2r-1}$ forms a transversal coalition with $P_3$, $P_6$, $P_8$, $P_{10}$, $\ldots$,$P_{2r}$. Here, each member in the $trc$-partition is a transversal coalition partner to exactly $(r-1)$ other members. We can partition $u_i$ into at most four sets, while the remaining $(r-2)$ vertices of each edge can be partitioned into two sets to obtain the maximum-order $trc$-partition. Therefore, $C_\tau(C_n)=2(r-2)+4=2r$.\\
{\bf Case 2:} If $n$ is odd, then we get a $trc$-partition as $\varphi=\{ P_1, ~P_2, ~P_3, ~P_4, ~P_5, ~P_6,\\ \ldots, ~P_{3r-5}, ~P_{3r-4}, ~P_{3r-3}\}=\{ \{u_1\}, \{u_2,u_4,\ldots, u_{n-1}\},\{u_3,u_5,\ldots, u_{n}\}, \{ x_{1,1}\},\\ \{ x_{2,1}, x_{3,1}, x_{4,1}, \ldots, x_{n-1,1}\}, \{x_{n,1}\}, \{ x_{1,2}\}, \{ x_{2,2}, x_{3,2}, x_{4,2},\ldots, x_{n-1,2}\}, \{x_{n,2}\},\\\ldots, \{ x_{1,r-2}\}, \{ x_{2,r-2}, x_{3,r-2},\ldots, x_{n-1,r-2}\}, \{x_{n,r-2}\}\}$. Here, $P_1$ forms a transversal coalition with $P_2$, $P_3$, $P_5$, $P_8$, $P_{11}$,$\ldots$, $P_{3r-4}$; $P_2$ forms a transversal coalition with $P_1$, $P_6$, $P_9$, $P_{12}$, $P_{15}$,$\ldots$, $P_{3r-3}$; $P_3$ forms a transversal coalition with $P_1$, $P_4$, $P_7$, $P_{10}$, $P_{13}$,$\ldots$, $P_{3r-5}$. Observe that $P_1$ forms a transversal coalition with $r$ members in the $trc$-partition. In addition, $P_2$ and $P_3$ form a transversal coalition with exactly $(r-1)$ members in the $trc$-partition. We can partition $u_i$ into at most three sets, while the remaining $(r-2)$ vertices of each edge can be partitioned into three sets to obtain the maximum-order $trc$-partition. Therefore, $C_\tau(C_n)=3+3(r-2)=3(r-1)$.\\
\indent Suppose that we partition $P_2$ as $P_{2,1}$ and $P_{2,2}$ for $n$ odd, then we get the new coalition partition as $\varphi'=\{ P_1, ~P_{2,1},~P_{2,2}, ~P_3, ~P_4, ~P_5, ~P_6,\ldots, ~P_{3r-5}, \\~P_{3r-4}, ~P_{3r-3}\}$, where $P_{2,1}=\{ u_2, u_6,\ldots, u_{n-3}\}$ and
$P_{2,2}=\{ u_4, u_8,\ldots, u_{n-1}$\}. $P_{2,1}$ and $P_{2,2}$ are not transversal coalition partners because there exists an edge $e_n=\{u_n, u_1, x_{n,1},\ldots, x_{n,r-2}\}$ whose intersection with $P_{2,1}\cup P_{2,2}$ is empty. Furthermore, $P_6$ does not form a transversal coalition with $P_{2,1}$ or $P_{2,2}$. Observe that $P_{2,1}$, $P_{2,2}$ and $P_6$ do not form a transversal coalition with any other set in $\varphi'$. If we consider that $P_6$ forms a coalition with $P_5$, then the intersection of edge $e_1$ with $P_5\cup P_6$ is empty. Therefore, we need to consider the set $P_5\cup P_6$, which forms a coalition with $P_4$. Similarly, we need to consider the union of sets $P_8\cup P_9,\ldots, P_{3r-4}\cup P_{3r-3}$. As a result, $C_\tau(C_n)< 3(r-1)$. Hence, $3(r-1)$ is the maximum order in the $trc$-partition of $C_n$ for $n$-odd. Similarly, if we partition any other set in $\varphi'$, then we can easily show that $C_\tau(C_n)< 3(r-1)$. Also, for $n$ even, we can show that $2r$ is the maximum order in the $trc$-partition. Therefore, $C_\tau(C_n)=3(r-1)$ if $n$ is odd and $C_\tau(C_n)=2r$ if $n$ is even.
\end{proof}	
\begin{example}
Consider a linear cycle of length $5$ in an $5$-uniform hypergraph as $C_5 = u_1 e_1 u_2e_2u_3e_3u_4e_4u_5e_5u_1$
with edge set
$\mathcal{E}=\{e_1 = \{u_1, u_2, x_{1,1},x_{1,2}, x_{1,3}\}, \\
~e_2=\{u_2, u_3, x_{2,1}, x_{2,2}, x_{2,3}\},
~ e_3=\{u_3, u_4,  x_{3,1},x_{3,2}, x_{3,3}\},
~ e_4=\{u_4, u_5, x_{4,1},x_{4,2},\\ x_{4,3}\},
~ e_5=\{u_5, u_1, x_{5,1},x_{5,2}, x_{5,3}\}\}$.
The $trc$-partition is: 
$\varphi = \{ P_1=\{u_1\}, ~P_2=\{u_2, u_4\}, ~P_3=\{u_3, u_5\}, ~P_4=\{x_{1,1}\}, ~P_5=\{x_{2,1}, x_{3,1}, x_{4,1}\}, ~P_6=\{x_{5,1}\}, ~P_{7}=\{x_{1,2}\}, ~P_{8}=\{x_{2,2}, x_{3,2}, x_{4,2}\}, ~P_{9}=\{x_{5,2}\}, ~P_{10}=\{x_{1,3}\}, ~P_{11}=\{x_{2,3}, x_{3,3}, x_{4,3}\},\\ ~P_{12}=\{x_{5,3}\} \}$. Here, $P_2$ forms a transversal coalition with $P_1, P_6, P_9, P_{12}$;  $P_3$ forms a transversal coalition with $P_1, P_4, P_7,P_{10}$; Also, $P_5$, $ P_8, P_{11}$ forms a transversal coalition with $P_1$. The transversal coalition number $C_\tau(C_5) = 12.$
\end{example}
\section{Conclusion}
This paper investigates transversal coalitions in $r$-uniform hypergraphs. The transversal coalition number is determined using a constructive approach. By computing transversal coalition numbers for small values of $n$, a recurring pattern is identified and generalized to all $n$. In Section 1, the transversal coalition numbers are established for specific $r$-uniform hypergraphs, including the complete $r$-uniform hypergraph, the complete bipartite $r$-uniform hypergraph, $r$-uniform stars, and the complete $r$-partite $r$-uniform hypergraph. In Section 2, the exact values of the transversal coalition numbers for $r$-uniform linear paths and cycles are also determined.\\ 
\indent Several research problems remain open in this area. Some of these are outlined below.
\begin{enumerate}[label=\arabic*.]
\item Investigate the transversal coalition for a hypertree.\\
\item Study the transversal coalition for various products of $r$-uniform hypergraphs.\\
\item The concept of transversal coalitions can be extended to cost-effective resource allocation and monitoring problems.     \end{enumerate}      

\vspace{0.cm}

\vspace{0.3cm}

    \end{document}